\documentclass[10pt]{article}

\usepackage[utf8]{inputenc}
\usepackage[T1]{fontenc}  

\usepackage{amssymb}
\usepackage{amsmath}
\usepackage{amsthm}
\usepackage{amscd}
\usepackage{url}
\usepackage{nicefrac}
\usepackage{xcolor}

\usepackage{mathrsfs}

\usepackage{tcolorbox}
\usepackage{graphicx}

\usepackage{stmaryrd}
\usepackage{mathrsfs}
\usepackage{amscd}
\usepackage{bbm}
\usepackage{enumitem}
\usepackage{fncylab}
\usepackage{xcolor}
\usepackage{url}
\usepackage{leftidx}
\usepackage{calc}
\usepackage[width=\textwidth]{caption}
\usepackage[normalem]{ulem}
\usepackage{dsfont}

\definecolor{boubelcolor}{rgb}{.65,0.05,0}

\DeclareUnicodeCharacter{2260}{\colorlet{memoireboubel}{.}\color{boubelcolor}} 
\DeclareUnicodeCharacter{00B1}{\color{memoireboubel}{}} 
\DeclareUnicodeCharacter{00A3}{\colorlet{memoirejuillet}{.}\color{orange}} 
\DeclareUnicodeCharacter{00A7}{\color{memoirejuillet}{}} 
\theoremstyle{plain}
\newtheorem{them}{Theorem}[section]
\newtheorem{pro}[them]{Proposition}
\newtheorem{cor}[them]{Corollary}
\newtheorem{lem}[them]{Lemma}

\theoremstyle{definition}
\newtheorem{defi}[them]{Definition}

\theoremstyle{remark}

\DeclareMathOperator{\Jac}{|Jac|}

\newcommand{\E}{\mathcal{E}}

\newcommand{\R}{\mathbb{R}}

\newcommand{\N}{\mathcal{N}}

\newcommand{\dd}{\mathrm{d}}
\newcommand{\eps}{\varepsilon}

\newcommand{\X}{\mathfrak{X}}

\newcommand{\I}{\mathcal{I}}
\newcommand{\M}{\mathcal{M}}

\renewcommand{\d}{\mathsf{d}}

\title{SubRiemannian structures do not satisfy Riemannian Brunn--Minkowski inequalities.}
\author{Nicolas Juillet}
\date{}

\providecommand{\keywords}[1]{\textbf{\textit{Keywords---}} #1}
\providecommand{\MSC}[1]{\textbf{\textit{MSC (2010)---}} #1}

\newcommand{\Addresses}{{
  \bigskip
  \footnotesize

\noindent
  Nicolas~Juillet, \textsc{IRMA -- UMR 7501 Universit\'e de Strasbourg et CNRS\\
    7 rue Ren\'e Descartes\\
    67000 STRASBOURG (FRANCE)}
\par\nopagebreak
\noindent
  \textit{E-mail address}: \texttt{nicolas.juillet@math.unistra.fr}

}}

\begin{document}
\maketitle

\begin{abstract}
We prove that no Brunn--Minkowski inequality from the Riemannian theories of curvature--dimension and optimal transportation can be satisfied by a strictly subRiemannian structure. Our proof relies on the same method as for the Heisenberg group together with new investigations by Agrachev, Barilari and Rizzi on ample normal geodesics of subRiemannian structures and the geodesic dimension attached to them.
\end{abstract}
\\
\MSC{51F99, 53B99, 28A75, 28C10}

\medskip

\section{Statement of the results}\label{page1}
We consider subRiemannian structures in the rank-varying setting of \cite[Chapter 5]{ABR}. This definition is the usual definition based on a metric tensor defined on a vector bundle $\mathbb{U}$ above the tangent space $TM$ to a smooth\footnote{meaning $\mathcal{C}^\infty$ as in the reminder of the paper.} connected manifold $M$ of dimension $n\geq 3$. We assume that the rank-varying distribution $x\mapsto \Delta_x$ satisfies the Hörmander condition, which permits us to define the subRiemannian distance $\d$. This makes $(M,\d)$ a metric space that we moreover assume to be complete. As a consequence of it the length space $(M,\d)$ becomes a geodesic space. Finally, in order to eliminate typically Riemannian behaviours our structure is assumed to be \emph{strictly} subRiemannian in the sense that the rank of the distribution $x\in M\to \mathrm{dim}(\Delta_x)$ is uniformly bounded by $n-1$. We finally obtain a geodesic metric measure space $(M,\d,\mu)$ after we let $\mu$ be a smooth positive volume form. 

As an example the reader may consider the rank-fixed setting of~\cite{Ri_book} where $M=(M,g)$ is a smooth complete connected Riemannian manifold of dimension $n\geq 3$ and  $\Delta$ a distribution of rank $k\leq n-1$ satisfying the bracket generating condition of H\"ormander. With the subRiemannian distance $\d$ obtained by restricting $g$ to $\Delta$, we obtain a complete (see Proposition 2.7 in \cite{Ri_book}) geodesic metric space. Note that in this example $\mu$ may be the Riemannian volume attached to $g$ but one may choose another measure with positive and smooth density.

We aim at finding two sets $A$ and $B$ with $\mu(A)\simeq \mu(B)$ and a ratio $r\in]0,1[$ such that the set $\{x\in M:\,\exists(a,b)\in A\times B,\,r^{-1}\d(a,x)=(1-r)^{-1}\d(x,b)=\d(a,b)\}$ that we call $M_r$ satisfies $\mu(M_r)\leq 2^{-2} \mu(A)$. The precise statement is given in Theorem \ref{thm:main}. It has consequences in the theories of Ricci curvature and optimal transportation through the fact that related Brunn-Minkowski inequalities as introduced in \cite{St2}, modeled on the Riemannian setting studied in \cite{CMS}, are contradicted by our construction (see Section \ref{sec:conseq}). In particular it implies that the curvature-dimension conditions $\mathsf{CD}(K,N)$ of Lott, Sturm and Villani \cite{LV1,St1, St2} can not be satisfied. This (negative) result was first proved for the Heisenberg group \cite{Ju_imrn} about 10 years ago as an extract of the author's doctoral thesis \cite{Ju_thesis}. The same principle of proof -- that is formally presented in \cite{Ju_kyoto} -- is used in the present paper together with the new lesson taken from Agrachev, Barilari and Rizzi's memoir \cite{ABR} that in any subRiemannian structure there exist normal minimizing curves that are ample at any time. This permits a construction similar to the one considered for the Heisenberg group to work again.

Before we can sate our theorem, let us recall some general definitions concerning metric spaces. A \emph{minimizing segment} of a metric space $(\X,\d)$ is a curve $\gamma:[t_0,t_1]\to (\X;\d)$ such that there exists $v\geq 0$ with $\d(\gamma(s),\gamma(t))=v|t-s|$ for every $s,\,t\in [t_0,t_1]$. A curve defined on an interval $I$ is a \emph{locally minimizing curve} if for $t\in I$ the restriction of $\gamma$ to some closed interval $[t-\eps,t+\eps]\cap I$ is a minimizing segment. A metric space is called \emph{geodesic} if for every $a,\,b\in \X$ there exists a minimizing segment from $a$ to $b$. In a geodesic space, a point $m$ is a \emph{mid point} of ratio $r$ (with $r\in [0,1]$) from $a$ to $b$ if there exists a geodesic segment $\gamma$ with $(\gamma(0),\gamma(r),\gamma(1))=(a,m,b)$. This notion extends to the \emph{mid set} of ratio $r$ from $A$ to $B$ that is the set of all mid points of ratio $r$ from some $a\in A$ to some $b\in B$.

\begin{them}\label{thm:main}
Let $M$ be a smooth, complete and connected subRiemannian structure of dimension $n\geq 3$ equipped with a  measure $\mu$ of smooth positive density, such that the varying rank of the distribution is smaller than $n-1$. We denote by $\mathcal{N}$ its minimal geodesic dimension. For every $\eps>0$, there exist Borel sets $A,\,B\subset M$ of finite non zero measure with $\mu(A)\in [\mu(B)(1-\eps),\mu(B)(1+\eps)]$ and a ratio parameter $r\in ]0,1[$ such that the mid set of ratio $r$ from $A$ to $B$ that we call $M_r$ satisfies $2^{\mathcal{N}-n}\mu(M_r)\leq \mu(B)(1+\eps)$. Moreover for every $R>0$ we can assume the diameter of $A\cup B$ to be smaller than $R$. 
\end{them}
Two complementary remarks are in order:\begin{itemize}
\item The mid set $M_r$ may be non measurable so that $\mu(M_r)$ has to be understood as the outer measure of $M_r$.
\item The \emph{minimal geodesic dimension} $\mathcal{N}$ is introduced in Definition \ref{def:geodim}. In particular $\mathcal{N}\geq n+2$.
\end{itemize}

Let us formally state the consequence on the Brunn-Minkowski inequality that we already mentioned in this introduction:
\begin{cor}\label{coro:coro}
SubRiemannian structures (as in Theorem \ref{thm:main}) do not satisfy any Brunn-Minkowski inequality $\mathsf{BM}(K,N)$, as well as any Curvature-Dimension condition $\mathsf{CD}(K,N)$ for $K\in \R$ and $N\geq 1$. 
\end{cor}
The definition of $\mathsf{BM}(K,N)$ is recalled in Section \ref{sec:conseq}.

The structure of the paper is as follows. In the first part we gather useful statements on special curves called normal geodesics. When restricted they can become minimizing segments. We introduce the growth vector and the geodesic dimension associated to a normal geodesic that is ample at some time $t$. We also recall the notion of smooth (pair of) points. In the second part, we prove Theorem \ref{thm:main}. The proof is based on Theorem 1 in \cite{Ju_kyoto}, the properties of ample normal geodesics and a few new ideas. As in \cite{Ju_kyoto} we need two special maps: the \emph{mid point map} of ratio $r$ and the \emph{geodesic inverse map} of ratio $r$ and choose $r$ so that the Jacobian determinant of the second map is $\pm1$. Important in the construction is that we can reverse the orientation so that we can assume $r\leq1/2$. In the last part we explain some relations with the contemporary literature on Ricci curvature for subRiemannian structures and metric spaces.

\paragraph{Aknowledgements} I would like to thank the organizers of the thematic trimester \emph{Geometry, Analysis and Dynamics on Sub-Riemannian Manifolds} held  in Paris at IHP in 2014. The result of this paper was announced during the thematic day on \emph{Optimal transport and sub-Riemannian manifolds}, on October 15th. Before, I benefited of instructive discussions with Davide Barilari, Ludovic Rifford and Luca Rizzi. I also wish to thank Luca Rizzi for encouraging me to write this note. I thank Samuel Borza for carefully reading the preprint. Finally I thank the referee for very complete reports that permitted me to improve many aspects of the paper. Resulting from the editorial process it is worth noting that the proof presented here does not encompass the case of almost-Riemannian structures (for which the rank of the distribution equals the dimension on a small set). The position of $m$ (see p.\ \pageref{egal_un}) can namely not be prescribed but only the neighborhood of the curve where it lies. However, the case of the Gru\v sin plane is addressed and solved in \cite{Ju_kyoto,Ju_thesis}.

\section{Ampleness and smoothness}

The characterization of minimizing segments (or more commonly minimizing curves) in a subRiemannian structure is a complex problem with a long history (see \cite{Mt}). We will only focus on the \emph{normal geodesics} that are the most intuitive minimizing curves in the sense that these are the only ones occurring in Riemannian geometry. Take care however that some of the so-called \emph{strictly abnormal geodesics} may also be restricted to minimizing segments. As the latter are less understood we only use normal geodesics while implementing the construction in \cite{Ju_imrn, Ju_kyoto} to more general subRiemannian structures. The \emph{subRiemannian Hamiltonian} is a function $H$ defined in local coordinates on $T^*M$ by the formula $H(x,p)=\frac12\sum_{i=1}^{\mathrm{rank}(\Delta_x)} p(X_i)^2$ where $(X_i)_i$ is an orthonormal basis of $(\Delta_x,g)$. It generates an Hamitonian flow defined as usual through the Hamiltonian gradient $\vec{H}$ by the equation $\dot\Psi=\vec{H}\circ\Psi$, $\Psi(0)=(x,p)$ where $(x,p)\in T^*M$. For every $x\in M$ we introduce the subRiemannian exponential map $\mathcal{E}_x:p\in T^*_x M\to \mathcal{E}_x(p)=\pi[\Psi(t)]\in M$ where $\pi:T^*M\to M$ is the canonical projection to the basis and $\Psi$ is defined with initial conditions $(x,p)$. For any $(x,p)\in T^*M$, the curve $t\mapsto \mathcal{E}_{x}(t p)$ is called a \emph{normal geodesic} and $\Psi$ a normal extremal. Therefore a normal geodesic is a curve in $M$ that admits a normal extremal lift. Here is the relation between normal geodesics and minimizing segments, and hence with our problem.

\begin{lem}[{Theorem 1.14 in \cite{Mt}}]\label{lem:shorten}
For every $(x,p)\in T^*M$ and $t_0\in \R$, there exists $\eps>0$ such that $t\mapsto \mathcal{E}_x(tp)$ is a minimizing segment on $[t_0-\eps,t_0+\eps]$ and the unique one between its ends. Thus, normal geodesics are locally minimizing curves.
\end{lem}

The notion of normal curve can be reinforced in terms of strictly or strongly normal curves \cite[Definition 2.14]{ABR}, as well as ample curves. Precisely, a curve is strictly normal on $[s,t]$ if it does not admit an abnormal extremal lift on this interval. It is strongly normal on $[s,t]$ if is strictly normal on any interval $[s,t']\subset [s,t]$. Since we do not define abnormal curves in this note we rely on the alternative equivalent definitions provided by Proposition 3.6 in \cite{ABR}. To state them in Definition \ref{def:special} we need to introduce some more notation. Let $\gamma:[t_0,t_1]\to M$ be an absolutely continuous curve. In a neighborhood $U$ of its support we consider a family of vector fields $(X_i)$ such that $(X_i(x))_{i}$ spans $\Delta_x$ at every $x\in U$. We introduce now controls $u_i:[t_0,t_1]\to \R$ such that $\dot{\gamma}(t)=\sum u_i(t)X_i(\gamma(t))$ at almost every time $t$. Let $P_{s,t}$ be the diffeomorphism associated with the flow on $[s,t]$ of this controlled ODE for variable initial positions in a neighborhood of $x$, i.e $P_{s,s}(x)=x$ and $\dd P_{s,t}(x)/\dd t= u_i(t)X_i(P_{s,t}(x))$. Finally we set $\mathscr{F}_{\gamma,s}(t)=(P_{s,t})_*^{-1}\Delta_{\gamma(t)}$. We also introduce derivatives of $t\mapsto \mathscr{F}_{\gamma,s}(t)$ as follows (see also \S3.4 in \cite{ABR}):
\[\forall i\geq 1,\quad \mathscr{F}_{\gamma,s}^{i}(t)=\left\{\frac{d^j}{dt^j} v(t):\, v(t)\in \mathscr{F}_{\gamma,s}(t)\text{ smooth, }j\leq i-1\right\}\subset T_{\gamma(s)}M.\]

\begin{defi}[Special normal geodesics]\label{def:special}
Let $\gamma$ be a normal geodesic. It is said \emph{strictly normal} on $[s,t]$ if $\mathrm{span}\{\mathscr{F}_{\gamma,s}(s'):\,s'\in [s,t]\}=T_{\gamma(s)}M$. It is \emph{strongly normal} on $[s,t]$ if it is strictly normal on each $[s,t']\subset [s,t]$. The curve is said \emph{ample} at time $s$ if the family of derivatives of $t\mapsto \mathscr{F}_{\gamma,s}(t)$ at time $s$ spans the whole $T_{\gamma(s)}M$.
\end{defi}


According to Proposition 3.6 of \cite{ABR} any ample curve at time $s$ is strongly normal on $[s,t]$. One of the most important results of \cite{ABR} is the following:
\begin{pro}[{Theorem 5.17 in \cite{ABR}}]\label{pro:exists}
For every $x\in M$, there exists $T>0$ and a normal geodesic $t\mapsto \mathcal{E}_x(tp)$ defined on $[0,T]$ that is ample at every time.
\end{pro}

We now define  the geodesic dimension attached to the subRiemannian exponential map for a covector $p\in T_x^*M$. For this we introduce the \emph{growth vector} $\mathcal{G}_p=\{k_1,\ldots, k_m\}$ where $k_i=\dim( \mathscr{F}_{\gamma,0}^{i}(0))$ and $m=\min\{i\in \mathbb{N}:\,k_{i}=n\}$ are given for the normal curve $\gamma: t\in\R\mapsto \E_x(tp)$. In particular $\gamma$ is ample at $0$ if and only if $m$ is finite. Notice also the relation $k_1=\mathrm{dim}(\Delta_x)$.
\begin{defi}[Geodesic dimension]\label{def:geodim}
Let $x$ be a point of $M$. For every covector $p$ in $T^*_x M$ we define $\mathcal{N}_{x,p}$ as $\mathcal{N}_{x,p}=\sum_{i=1}^m (2i-1)(k_i-k_{i-1})$ if $t\mapsto \mathcal{E}_x(tp)$ is ample at time $0$ and $+\infty$ if it is not. The \emph{geodesic dimension} at $x\in M$ is $\mathcal{N}_x=\min\{\mathcal{N}_{x,p}:\,p\in T^*_x M\}<\infty$ (the finiteness comes from Proposition \ref{pro:exists}). Finally, we denote by $\mathcal{N}=\min_{x\in M} \mathcal{N}_x$ the \emph{minimal geodesic dimension} on $M$.
\end{defi}

As explained in Proposition 5.49 of \cite{ABR}, for every $x\in M$, the geodesic dimension at $x$ is greater than the homogeneous dimension (or Hausdorff dimension) of the tangent Carnot group. This one is computed with Mitchell's formula from the dimensions of the spaces $\Delta_x,\, (\Delta+[\Delta,\Delta])_x,$\ldots  and hence it is larger than $n+(n-r)=r+2(n-r)$ where $r$ is the maximal rank of $\Delta$. Using the same principle with the formula defining $\mathcal{N}_{x,p}$ we obtain
\[\mathcal{N}_{x,p}= 1 k_1+3k_2+5k_3+\ldots +(2m-1)k_m\geq r+3(n-r).\]
As $M$ is a strictly subRiemannian structure it holds $r<n$. Thus $\mathcal{N}_{x,p}\geq n+2$.

\begin{defi}[Pairs of smooth points]\label{def:smooth}
Let $x$ and $y$ be points of $M$. The pair $(x,y)$ is a \emph{pair of smooth points}  or a \emph{smooth pair} if $\d^2$ is $\mathcal{C}^2$ at $(x,y)$. We denote by $\Sigma$ the set of smooth pairs and $\Sigma_x$ the set of the points $y\in M$ such that $(x,y)\in \Sigma$.
\end{defi}

Note that $\Sigma$ and $\Sigma_x$ are open. Moreover since $M$ is strictly subRiemannian the differentiability of $\mathsf{d}^2$ in Definition \ref{def:smooth} can be replaced by the one of $\mathsf{d}$, even at $x=y$. The following two results are taken from \cite{Agr,RT05}\footnote{In \cite{Agr} smooth points were in fact introduced through the (equivalent) statement on $\mathcal{E}_x$ appearing in Proposition \ref{pro:charac}.}.
\begin{pro}
The pair $(x,y)$ is a pair of smooth points if and only if $z\mapsto \d(x,z)^2$ is smooth at $y$, if and only if $z\mapsto \d(z,y)^2$ is smooth at $x$, if and only if $\d^2$ is smooth at $(x,y)$.
\end{pro}

\begin{pro}\label{pro:charac} 
Let $x,\,y$ be points of $M$. Then $y\in \Sigma_x$ if and only if there exists $p\in T^*_x M$ such that $t\mapsto \mathcal{E}_x(tp)$ is the unique minimizing segment from $x$ to $y$ defined on $[0,1]$ and $\mathcal{E}_x$ is a submersion at $p$.
\end{pro}

As stated in the following lemma it is possible to restrict an ample normal geodesic to a piece where all pairs of points are smooth pairs.

\begin{lem}\label{lem:restr}
Let $\gamma:[t_0,t_1]\to M$ be a normal geodesic that is ample at every time $t\in [t_0,t_1]$. Then for every $s\in [t_0,t_1]$ there exists $\eps>0$ such that $(\gamma(s),\gamma(t))$ is a pair of smooth points for all $t\in [t_0,t_1]$ such that $0<|t-s|< \eps$.
\end{lem}
\begin{proof}
 We fix $s\in [t_0,t_1]$. The first basic remark is that being a normal geodesic is conserved while reversing the time. This comes from the definition of the normal extremal lift as an Hamiltonian flow. The same remark is true for the reinforced properties strictly normal and ample as can be checked in \cite{ABR}. Therefore it is enough to find $\eps>0$ such that the result holds for all $t\in [t_0,t_1]$ satisfying $0<t-s<\eps$. We can use Lemma \ref{lem:shorten} to restrict the curve so that it is the unique minimizing segment between its ends $\gamma(s)=x$ and $\gamma(t)=\mathcal{E}_x(p)$. In order to find smooth pairs we will apply Proposition \ref{pro:charac}. Therefore we want to prove that $\mathcal{E}_x$ is a submersion at $p$. Recall that $\gamma$ is ample at every point and that hence it is strongly normal. Hence we can follow the argument of \S 3 in the short paper  \cite{Agr} on smooth points: The ``end-point mapping'' is a submersion on each segment $[s,t]$. Due to the theory of conjugate points\footnote{The referee suggested  \cite[Appendix A]{BR} as a more accessible reference on the theory of conjugate points on subRiemannian structures. In a setting including ours, Theorem 72  states that on a minimizing segment there is at most one conjugate point to the first extreme point namely, if it exists, the second extreme point. Lemma 71 provides an alternative simpler proof of Lemma \ref{lem:restr} (the part on smooth points) that contrarily to Theorem 72 is not based on minimality. This alternative proof is based on the fact that if conjugate points to $\gamma(s)$ would accumulate at $\gamma(s)$, the curve restricted to a small segment would be abnormal, a contradiction with the ampleness at every time.} recalled in \cite{Agr}, $\mathcal{E}_{\gamma(s)}$ is also a submersion if and only if $(\gamma(s),\gamma(t))$ are not ``conjugate points''.  According to this theory there is at most one $\bar{t}$ such that $\gamma$ is a minimizing segment on $[s,\bar{t}]$ and $(\gamma(s),\gamma(\bar{t}))$ are conjugate points. Therefore the result holds for $\eps=\bar{t}-s$.
\end{proof}

On $\Sigma\times [0,1]$ we define the \emph{mid point map} $\M$ by $\M(x,y,t)=\mathcal{E}_x(tp)$ where $p$ satisfies $y=\mathcal{E}_x(p)$. It is a smooth map on $\Sigma \times ]0,1[$ as can be seen from the fact that $\varphi:(x,p)\to (x,\mathcal{E}_x(p))$ is a local diffeomorphism at points $(x,p)$ such that $\varphi(x,p)\in \Sigma$. The mid point map satisfies $\M(x,y,0)=x$ and $\M(y,x,t)=\M(x,y,1-t)$. We denote by $\M^t$, $\M^t_{x,.}$ and $\M^t_{.,y}$ the smooth maps defined on $\Sigma$, $\Sigma_x$ and $\Sigma_y$ by $\M^t(x,y)=\M_{x,.}^t(y)=\M_{.,y}^t(x)=\M(x,y,t)$, respectively.

\begin{pro}[Contraction rate, {Lemma 6.27 in \cite{ABR}}]\label{pro:cont}
Let $(x,y)$ be in $\Sigma$ and $p\in T^*_x M$ be such that $y=\mathcal{E}_x(p)$. Assume that the minimizing segment from $x$ to $y$ is included in a single coordinate chart $\{x_i\}_{i=1}^n$. Then there exists a positive $C$ such that $\Jac(\M^t_x)(y)\sim C t^{\mathcal{N}_{x,p}}$ as $t$ goes to zero.
\end{pro}

We define also the \emph{inverse geodesic map} by $\I(m,y,t)=\mathcal{E}_m([-t/(1-t)]p)$ where $y=\mathcal{E}_m(p)$, and note it also $\I_m^t(y)$. It is defined for $(m,y)\in \Sigma$ and every $t\neq 1$. Formally the parameter $-t/(1-t)$ is taken such that $\M^t(\I_m^t(y),y)=m$. For this equation to be correct it is sufficient that $\I_m^t(y)\in \Sigma_y$ and $s\in [-t/(1-t),1]\to \mathcal{E}_m(s p)$ is the unique minimizing segment between its ends.
\section{Proof of Theorem \ref{thm:main}}

We prove Theorem \ref{thm:main} after Theorem 1 in \cite{Ju_kyoto} whose proof is based on estimates of the Jacobian determinants of $\M^r_a$ and $\I^r_m$. 
\paragraph{Jacobian determinant of the mid point map}
Let us start with an estimate of the contraction rate along a special normal geodesic. Its length will be chosen short enough to obtain a rate close to a power law.
\begin{pro}\label{pro:jac}
For every $x\in M$, $R>0$ and $\eps>0$ there exists a minimizing segment $\gamma:[s,t]\to M$ such that
\begin{itemize}
\item the range of $\gamma$ is in the ball of center $x$ and radius $R$,
\item $\gamma$ is a subsegment of a normal geodesic, ample at every point,
\item for every $u\in ]s,t[$ the pairs $(\gamma(s),\gamma(u))$, $(\gamma(u),\gamma(t))$ and $(\gamma(s),\gamma(t))$ are smooth,
\item for every $r\in [0,1]$ it holds
\begin{align*}
\left\{
\begin{aligned}
\Jac(\M_{a,\cdot}^r)(b) \leq \left(1+\eps\right)^2r^{\N}\\
\Jac(\M_{b,\cdot}^r)(a) \leq \left(1+\eps\right)^2r^{\N}
\end{aligned}
\right.
\end{align*}
\end{itemize}
where $a=\gamma(s)$ and $b=\gamma(t)$. We recall that $\N$ is the geodesic dimension.

 \end{pro}

\begin{proof}

Let $x$ be a point of $M$ and $R>0$, $\eps>0$ as in the statement. Proposition \ref{pro:exists} provides $p_0\in T^*_x M$ such that $\gamma:t\mapsto \mathcal{E}_x(t p_0)$ is a normal geodesic that is ample at any point of $[0,T_0]$.  We restrict the interval to $[0,T]\subset [0,T_0]$ so that the support of the curve $\gamma([0,T])$ is in the ball of center $x$ and radius $R$.

For some $s\in ]0,T[$ we set $a=\gamma(s)$. Lemma \ref{lem:restr} provides $t\in [0,T]\setminus \{s\}$ such that $\gamma(u)\in \Sigma_a$ for every $u\in ]s,t]$ (or $u\in [t,s[$ if $t<s$). We set $y=\gamma(t)=\mathcal{E}_a(p)$. Proposition \ref{pro:cont} yields
\begin{align*}
\Jac(\M_{a,.}^r)(y)\sim_{r\to 0^+} C r^{\N_{a,p}}.
\end{align*}
Let $\alpha\in [0,1]$ be such that for every $r\in [0,\alpha]$,
\begin{align*}
C(1+\eps)^{-1}r^{\N_{a,p}}\leq \Jac(\M^r_{a,.})(y) \leq C(1+\eps)r^{\N_{a,p}}.
\end{align*}
Let $b$ be $\gamma(s+\alpha (t-s))=\mathcal{E}_a(\alpha p)=\M^\alpha_{a,.}(y)$. For every $r\in [0,1]$ we have $\M^r_{a,.}=\M^{\alpha r}_{a,.}\circ (\M^\alpha_{a,.})^{-1}$. Hence, differentiating at $b$ we obtain
\begin{align*}
\Jac(\M^r_{a,\cdot})(b)=\Jac(\M^{\alpha r}_{a,\cdot})(y)\cdot\Jac(\M_{a,\cdot}^{\alpha})(y)^{-1}.
\end{align*}
For every $r\in [0,1]$, it implies
\begin{align}\label{eq:N}
\Jac(\M_{a,\cdot}^r)(b) \leq \left(1+\eps\right)^2r^{\N_{a,p}}.
\end{align}
Note that this estimate is still correct for $b=\gamma(u)=\M^{\alpha'}_{x,.}(y)$ where $u=s+\alpha'(t-s)$ and $\alpha'<\alpha$. Moreover it does not depend on the relative position of $a=\gamma(s)$ and $y=\gamma(t)$. Therefore, for every $s\in ]0,T[$ there exists positive functions $\eta$ and $\zeta$ such that the minimizing segments parametrized on $[s-\eta(s),s]$ and $[s,s+\zeta(s)]$ both let appear an estimate corresponding to \eqref{eq:N}, with $\gamma(s)$ in the role of $a$, point $\gamma(u)$ with $u\neq s$ such that $-\eta(s) \leq u-s\leq \zeta(s)$ in the role of $b$ , and possibly another value for $\mathcal{N}_{a,p}$. With Lemma \ref{lem:restr} we have also made it possible to assume that $(\gamma(s),\gamma(u))$ is a smooth pair. 

Thanks to these remarks and Lemma \ref{lem:N} below we can finally assume that on some minimizing subsegment $\gamma([s,t])$ of $\gamma([0,T])$ equation \eqref{eq:N} is not only satisfied for $a=\gamma(s)$ together with the contraction family $(\M^r_{a,.}(b))_{r\in [0,1]}$, but also for $b=\gamma(t)$ and $(\M^r_{b,.}(a))_{r\in [0,1]}$. Moreover for every $m$ strictly in the geodesic segment $[a,b]\subset U$ all the pairs $(a,m)$, $(m,b)$ and $(a,b)$ are smooth.
\end{proof}

The following lemma is used in the proof of Proposition \ref{pro:jac}. We found it interesting enough to isolate it.

\begin{lem}\label{lem:N}
Let $\eta$ and $\zeta$ be two positive functions on $]0,T[$. Then there exists $[s,t]\subset ]0,T[$ (with for instance $s<t$) such that $t\in]s,s+\zeta(s)[$ and $s\in ]t-\eta(t),t[$ (so that $[s,t]\subset ]t-\eta(t),s+\zeta(s)[$). 
\end{lem}
\begin{proof}

Consider $F_i=\{s: \zeta(s)>1/i\}$. There exists an integer $i$ such that the cardinal of $F_i$ is non countable. Introduce now $F_{i,j}=\{t\in F_i: \eta(t)>1/j\}$. It is also non countable if $j$ is great enough. Therefore the set $F_{i,j}$ has a point $u\in]0,T[$ with accumulation on both sides. We conclude by choosing $t\in F_{i,j}\cap ]u,u+1/2k[$ and $s\in F_{i,j}\cap ]u-1/2k,u[$ where $k\geq \max(i,j)$.
\end{proof}

\paragraph{Jacobian determinant of the geodesic inverse} \label{egal_un}
We derive here an estimate for the geodesic inverse that enhances the results obtained in the previous paragraph. Therefore, let $a=\gamma(s)$ and $b=\gamma(t)$ be as obtained in Proposition \ref{pro:jac} and let $m=\M(a,b,r)$ be a mobile point moving with $r\in [0,1]$. From $\M^r(\I_m^r(y),y)=m$ that we write for $y$ in a neighborhood of $b$ the differentiation at point $b$ yields
\begin{align*}
\Jac \M^r_{\cdot,b}(a) \Jac\I_m^r(b)=\Jac\M^r_{a,\cdot}(b).
\end{align*}
Recall that $\M_{.,b}^r=\M_{b,.}^{1-r}$. As $r$ tends to $1$, $\Jac(\M^r_{.,b})(a)$ and $\Jac(\M^r_{a,.})(b)$ tend to zero or one, respectively. As $r$ tends to $0$ the limits are the same but in the other order. Hence $r\mapsto \Jac\I^r_m(b)$ has limits $0$ and $+\infty$ at the bounds of  $]0,1[$. As it is a continuous function of $r$ the intermediate value theorem yields 
\begin{align}\label{eq:O}
\Jac\I_{m}^r(b)=1
\end{align}
for some $r\in]0,1[$.

\paragraph{Geometric construction}

Before the proof of Theorem \ref{thm:main} we briefly recall the principle of proof in \cite{Ju_kyoto} that together with \eqref{eq:N} and \eqref{eq:O} exhibits sets $A$ and $B$. The present proof shall be slightly different but both proofs share the same geometric meaning. The sets $A$ and $B$ may be seen as two small pearls at the end of a minimizing segment with ends $a,\,b$ and point $m=\M^r(a,b)$ selected according to the previous paragraphs. An essential conclusion of Proposition \ref{pro:jac} is that we can assume $r\leq 1/2$ (if $r>1/2$, using the fact that our statements are symmetric in $a$ and $b$ we swap the labelling of $a$ and $b$ so that $r$ is replaced by $1-r\leq 1/2$). In the coordinate chart the set $B=\mathcal{B}(b,\rho)$ is a (small) Euclidean ball of center $b$ and radius $\rho$ and $A=\I_m^r(B_\rho)$ is its image by $\I_m^r$. Radius $\rho$ tends to zero so that we can analyse volumes at first order as if they were the images of affine maps. In particular, since $\I_m^r$ has Jacobian $\pm1$ the ratio $\lambda_n(A)/\lambda_n(B)$ tends to 1 as $\rho$ goes to zero.  For the same reason the set $A=\I^r_m(B)$ looks like an ellipso\"id and the mid set $\M^r(A,B)$ looks like an ellipso\"id as well, the volume of which is smaller than $r^\mathcal{N}2^n\lambda_n(B)\leq 2^{n-\mathcal{N}}\lambda_n(B)$.  The geodesic dimension $\N$ is a lower bound on the power law appearing in the contraction of center $a$ and ratio $r$ whereas $2^n$ comes from the fact that at first order the mid point of $(p,q)\in A\times B$ is determined by $q-\I_m^{1-r}(p)\in \mathcal{B}(0,2\rho)$ a ball whose volume is $2^n\lambda_n(B_\rho)$.


\begin{proof}[Proof of Theorem \ref{thm:main}]
Let $\eps>0$ be a positive parameter. We decompose it introducing $\eps_1,\,\eps_2\in ]0,\eps]$ such that  $(1+\eps_1)(1+\eps_2)\leq 1+\eps/2$. Let $x$ be a point of $M$ and $U$ be a neighborhood of $x$ small enough to be contained in a single coordinate chart $\{x_i\}_{i=1}^n$ where the Lebesgue mesure $\lambda_n$ satisfies $\left(\sup_U d\mu/d\lambda_n\right) \left(\sup_U d\lambda_n/d\mu\right)< 1+\eps_1$. Thus the estimates for $\mu$ and $\eps$ in Theorem \ref{thm:main} are satisfied if we manage to prove the following ones for $\lambda_n$ and $\eps_2$, namely $\lambda_n(A)\in [(1-\eps_2)\lambda_n(B),(1+\eps_2)\lambda_n(B)]$ and $2^{\mathcal{N}-n}\lambda_n(M_r)\leq \lambda_n(B)(1+\eps_2)$. 

Let $R$ be a radius such that the ball of center $x$ and radius $R$ is contained in $U$. The sets $A$ and $B$ will be subsets of the ball. Note that $R$ can be arbitrarily reduced to satisfy the last statement in Theorem \ref{thm:main} on the size of $\mathrm{diam}(A\cup B)$. We apply Proposition \ref{pro:jac} and the conclusion of the previous paragraph ``Jacobian determinant of the geodesic inverse'' to obtain three points $a,b,m\in\mathcal{B}(x,r)\subset U$ (for some $r\leq 1/2$, see the paragraph before the proof) that satisfy the properties listed there, in particular  \eqref{eq:N} that we take for $(1+\eps_2/2)$ in place of $(1+\eps)^2$, and \eqref{eq:O}.

In the coordinate chart we set $B_\rho=\mathcal{B}(b,\rho)$ the (small) Euclidean ball of center $b$ and radius $\rho$ and $A_\rho=\I_m^r(B_\rho)$ its image by $\I_m^r$. Since $\I_m^r$ has Jacobian $\pm1$ the ratio $\lambda_n(A_\rho)/\lambda_n(B_\rho)$ tends to 1 as $\rho$ goes to zero. In particular with respect to the first statement of Theorem \ref{thm:main} it is in $[1-\eps,1+\eps]$ for any $\rho$ small enough.

We set $F:(p,q)\in B_\rho\times B_\rho\mapsto \mathcal{M}^r(\mathcal{I}_m^r(p),q)$. Note $\mathcal{I}_m^r(B_\rho)=A_\rho$ and $F(B_\rho,B_\rho)=M_r$. The first order Taylor formula of $F$ at $(b,b)$ yields
\begin{align}\label{eq:clef}
\|(F(p,q)-m)-DF(b,b)\cdot(p-b,q-b)\|
\leq o(\|p-b\|+\|q-b\|)
\end{align}
where $(p,q)$ goes to $(b,b)$. As $F(p,p)=m$ for every $p\in B_\rho$ we get $DF(b,b)\cdot(h,h)=0$ for every vector $h$. Differentiating $F$ with the chain rule, this yields $D(\M^r)_1(a,b)\circ D\mathcal{I}_m^r(b)+D(\M^r)_2(a,b)=0$. Therefore, from \eqref{eq:clef} we derive
\begin{align*}
\|(\M^r(\mathcal{I}_m^r(p),q)-m)-D(\M^r)_2(a,b)\cdot(q-p)\|\leq o(\|p-b\|+\|q-b\|).
\end{align*}
As $(p,q)$ runs into $B_\rho\times B_\rho$, the pair $(\mathcal{I}_m^r(p),q)$ describes $A_\rho \times B_\rho$ and the vector $q-p$ describes $\mathcal{B}(0,2\rho)$. With the last equation we see that 
\[\M^r(\mathcal{I}_m^r(p),q)\in m+D(\M^r)_2(a,b)\cdot\mathcal{B}(0,2\rho)+\mathcal{B}(0,f(\rho))\]
where $f(\rho)=o(\rho)$ as $\rho$ goes to zero. For the Lebesgue measure $\lambda_n$ the volume on the right is equivalent to the one of $D(\M^r)_2(a,b)\cdot\mathcal{B}(0,2\rho)$ that, due to \eqref{eq:N}, is smaller than $2^nr^\N(1+\eps_2/2)\lambda_n(\mathcal{B}(0,\rho))$. For $\rho$ small enough this yields $\lambda_n(F(B_\rho,B_\rho))\leq 2^{\N-n}(1+\eps_2)\lambda_n(B_\rho)$. As we noticed above, this it is enough to conclude the proof.

\end{proof}

\section{Brunn--Minkowski inequalities and spaces with Ricci curvature bounded from below}\label{sec:conseq}

In this section we extend the meaning of $\M^t$ to a geodesic metric space $(\X,\d)$ and make it a multi-valued map.
\[\M^t(a,b)=\{m\in \X:\,\d(a,m)=t\d(a,b)\text{ and }\d(m,b)=(1-t)\d(a,b)\}\]
We denote by $\M^t(A,B)$ the set $\M^t(A,B)=\bigcup_{a\in A,\,b\in B}\M^t(a,b)$. Let $\mu$ be a Borel measure on $\X$ and $K\in \R$,\;$N\geq 1$ be curvature and dimension parameters. The metric measure space is said to satisfy the Brunn--Minkowski inequality of parameters $(K,N)$ if for any Borel sets $A$ and $B$ and $t\in[0,1]$ the measures of $A$, $B$ and $\M^t(A,B)$ satisfy the following inequality
\[\mu(\M^t(A,B))^{1/N}\geq \tau^{K,N}_{(1-t)}(\Theta)\mu(A)^{1/N}+\tau^{K,N}_t(\Theta)\mu(B)^{1/N}.\]
On the right-hand side if $K<0$, we have to take
\[
\tau^{K,N}_t(\Theta)=t^{1/N}\left(\frac{\sinh(t\Theta\sqrt{-K/(N-1)})}{\sinh(\Theta\sqrt{-K/(N-1)})}\right)^{1-1/N}
\]
where $\Theta=\sup_{(x,y)\in A\times B}\d(x,y)$. Other expressions are given for $K\geq 0$ but we only stress the fact that the Brunn--Minkowski  $\mathsf{BM}(K',N)$ implies $\mathsf{BM}(K,N)$  for every $K\leq K'$. 

Note that, for $K<0$, the coefficient $\tau^{K,N}_t(\Theta)$ tends to $t$ as $\Theta$ goes to zero. Thus, for $\Theta$ small enough, and $\mu(A)\simeq \mu(B)$ we see that $\mu(\M^t(A,B))$ should have its measure greater than or equal to $\mu(A)(1-\eps)$. This is independent from the dimension. Therefore, adapting the size parameter $R$ in Theorem \ref{thm:main} we obtain the corollary that $\mathsf{BM}(K,N)$ is not satisfied for any $K<0$ and thus any $K\in \R$.  For $K\geq 0$ we can alternatively apply verbatim the argument above since the limit $\tau^{K,N}_t(\Theta)\to_{\Theta\to 0} t$ also holds in this case. One can understand that Brunn--Minkowski inequalities are not satisfied from the fact that they were extrapolated from the Riemannian geometry. In particular $\mathsf{BM}(K,N)$ is satisfied by Riemannian manifolds with dimension lower or equal to $N$ and lower bound  $K$ on the Ricci curvature, by manifolds with generalized Ricci tensor greater than $K$, or by $\mathsf{RCD}(K,N)$ spaces, i.e. \emph{Riemannian curvature-dimension} spaces.  Hence the spaces $\X$ considered in the present paper are not $\mathsf{RCD}(K,N)$ spaces. This result is also proved in a recent paper by Huang and Sun \cite{HS} comparing the tangent spaces to $\mathsf{RCD}$ spaces with Carnot groups, the tangent spaces to subRiemannian structures. As stated in Corollary \ref{coro:coro} the intermediate property -- also implying $\mathsf{BM}(K,N)$ -- to be a $\mathsf{CD}(K,N)$ space is not satisfied as well. This can not be proved with the approach of Huang and Sun because these authors rely on infinitesimal Hibertianity, a property that is not required in $\mathsf{CD}(K,N)$ spaces. 

Inequalities $\mathsf{BM}(K,N)$ are not satisfied on subRiemannian structures. However, as recently proved by Balogh, Krist\'aly and Sipos other types of Brunn--Minkowski inequalities can be satisfied. This is proved with two different approaches in \cite{BKS1,BKS2}. These new inequalities involve sharp distortion coefficients adapted to the subRiemmannian geometries of the Heisenberg groups and of corank 1 Carnot groups -- notice that these coefficients also appeared in \cite{Ju_imrn,Riz} in the special case where $A$ is a single point. In an even more recent paper Barilari and Rizzi extend these Brunn--Minkowski inequalities to ideal subRiemannian manifolds, see \cite{BR}. In place of new distortion parameters one can alternatively relax the $\mathsf{CD}(K,N)$ inequalities through a simple multiplicative coefficient as Milman does in \cite{Mil} yielding ``quasi'' Brunn--Minkowski inequalities. Let us finally recall that modified versions of the $\mathsf{CD}(K,N)$ condition were proved for operators on special classes of subRiemannian structures \cite{BG}.

\bibliographystyle{abbrv}
\bibliography{basebib_killbm}
\Addresses
\end{document}